\documentclass[12pt]{article}
\usepackage[left=2.5cm,right=2.5cm,top=2.5cm,bottom=2.5cm,a4paper]{geometry}

\usepackage[a4paper]{geometry}
\usepackage{graphicx}
\usepackage{microtype}
\usepackage{siunitx}
\usepackage{booktabs}
\usepackage{graphics}
\usepackage{graphicx}
\usepackage{epsfig}
\usepackage{amsmath,amsfonts,amssymb,amsthm}
\usepackage{cleveref}
\usepackage{listings}
\usepackage{paralist}
\usepackage{sectsty}
\usepackage{datetime}
\usepackage{algorithmic,algorithm}
\usepackage[X2,T1]{fontenc} % or OT1 in place of T1

\usepackage{authblk}

\theoremstyle{definition}
\newtheorem{definition}{Definition}[section]
\newtheorem{remark}{Remark}

\theoremstyle{plain}
\newtheorem{theorem}[definition]{Theorem}
\newtheorem{proposition}[definition]{Proposition}

\newtheorem{corollary}[definition]{Corollary}
\newtheorem{lemma}[definition]{Lemma}

\newcommand{\N}{\mathbb Z_{\ge 0}}
\newcommand{\Ap}{\operatorname{Ap}}

\title{Frobenius Numbers Associated with Primitive Pythagorean Quadruples}
\author[1]{WonTae Hwang}
\author[2]{Kyunghwan Song\thanks{khsong@jejunu.ac.kr}}
\affil[1]{Department of Mathematics, and Institute of Pure and Applied Mathematics, Jeonbuk National University, Baekje-daero, Deokjin-gu, Jeonju-si, Jeollabuk-do, 54896 Republic of Korea}
\affil[2]{Department of Mathematics, Jeju National University, 102 Jejudaehakro Jeju, 63243, Republic of Korea}
\date{}

\begin{document}

\maketitle

\begin{abstract}
Let $(a,b,c,d) = \left(2mn, 2mp, m^2 - n^2 - p^2, m^2 + n^2 + p^2\right)$ be a primitive Pythagorean quadruple and let $S=\langle a, b, c, d\rangle$ be the numerical semigroup generated by $a,b,c,$ and $d.$
For convenience, we also let $Q = n^2 + p^2, \delta = \gcd(n,p),$ and $n = \delta n_0$ for some $n_0 \in \mathbb{Z}$. In this paper, we determine the Frobenius number of $S$ and derive an explicit formula in terms of $m$, $n$, and $p$ with the assumption that $m\geq 2Q$ for $\delta = 1$ and $m \geq \frac{2Q}{\delta} - 1$ for the remaining cases. The proof is based on an explicit complete residue system modulo $2mn$,
a normalization procedure for arbitrary semigroup elements, and a lift-orbit
description of boundary representatives.
\end{abstract}

\section{Introduction}
For positive integers $a_1,\ldots,a_k$ satisfying $\gcd(a_1,\ldots,a_k)=1,$ let $S=\langle a_1,\ldots,a_k\rangle$ denote the numerical semigroup generated by
$a_1,\ldots,a_k$. The \emph{Frobenius number} $g(S)$ is the largest integer that
does not belong to $S$. The problem of determining $g(S)$ from a given system of
generators is the classical Frobenius problem (see, for example,
\cite{Ramirez2005,Rosales2009}).

For a nonzero element $x\in S$, the Ap\'ery set of $S$ with respect to $x$ is $\text{Ap}(S,x)=\{s\in S:s-x\notin S\}.$
Equivalently, it consists of the least element of $S$ in each residue class modulo
$x$. In particular,
\begin{equation}\label{eq:apery-frob}
g(S)=\max\operatorname{Ap}(S,x)-x.
\end{equation}
This relation makes the Ap\'ery set one of the standard tools for studying
Frobenius numbers (see \cite{Selmer1977,Ramirez2009,
Rosales2009}).

When $S$ is generated by two relatively prime positive integers $a$ and $b$, its Frobenius number is given by the classical formula $g(\langle a,b\rangle)=ab-a-b$ of Sylvester \cite{Sylvester1883}.
For three or more generators, however, the situation becomes substantially more
complicated. Curtis \cite{Curtis1990} showed that, already for three generators,
the Frobenius number cannot in general be represented by a finite collection of
polynomial formulas. Moreover, the general Frobenius problem is NP-hard
\cite{Ramirez1996}. Nevertheless, several effective methods and explicit
results are known for numerical semigroups of embedding dimension three (see \cite{Rodseth1978,Rosales2011,Robles2012,Tripathi2017}).

The difficulty of the general problem has motivated the study of structured
families for which explicit Frobenius formulas can be obtained. The Ap\'ery set
has proved particularly useful in this direction. For instance, explicit
descriptions of Ap\'ery sets have led to formulas for the Frobenius numbers of
Thabit, repunit, and Mersenne numerical semigroups
\cite{Rosales2015,Rosales2016,Rosales2017}.
The classical Frobenius problem has also been generalized by taking the number
of representations into account. This leads to the $p$-Frobenius problem and
related notions such as the $p$-Ap\'ery set and the $p$-genus, with the classical
problem corresponding to $p=0$. These generalized Frobenius problems have been
studied for a variety of structured families \cite{Komatsu2022,Komatsu2023,Komatsu2024-r1,Komatsu2024-r12,Komatsu2024,Komatsu2025}.

Among structured numerical semigroups arising from Diophantine equations, those associated with Pythagorean triples are particularly relevant to this paper. If
$m>n$ are relatively prime positive integers of different parity,
then $(m^2-n^2,\,2mn,\,m^2+n^2)$ is a primitive Pythagorean triple.  Gil et al. \cite{Gil2015}
determined the Frobenius number of the numerical semigroup
generated by such a triple explicitly.  Further aspects of these semigroups have also been studied.  Mhanna \cite{Mhanna2020} described their
pseudo-Frobenius numbers and Ap\'ery sets, while
Elizeche and Tripathi \cite{Elizeche2020} studied numerical semigroups generated by primitive Pythagorean triplets, including their Frobenius numbers, genera, and pseudo-Frobenius numbers. Furthermore, Komatsu and Sury \cite{Komatsu2023} gave an explicit formula for the $p$-Frobenius number of primitive Pythagorean triples.

Pythagorean quadruples provide a natural higher-dimensional analogue of
Pythagorean triples. More generally, a Pythagorean $k$-tuple is an integer
solution of
$$x_1^2+\cdots+x_{k-1}^2=x_k^2.$$

The parametrization of Pythagorean quadruples and higher Pythagorean tuples has
also been studied from a Diophantine point of view. In particular, Frisch and
Vaserstein \cite{Frisch2012} established polynomial parametrizations
for Pythagorean quadruples and sextuples, together with a related parametrization
result for Pythagorean quintuples.

Motivated by these results, we consider the corresponding Frobenius problem for
primitive Pythagorean quadruples.  More precisely, let $m,n,p$ be
positive integers and consider
$$
(2mn,\,2mp,\,m^2-n^2-p^2,\,m^2+n^2+p^2).
$$
These integers satisfy
$$
(2mn)^2+(2mp)^2+(m^2-n^2-p^2)^2=(m^2+n^2+p^2)^2.
$$
Writing
$Q=n^2+p^2,$ we study the numerical semigroup
$$
S=\langle 2mn,2mp,m^2-Q,m^2+Q\rangle
$$
associated with this quadruple.  The quadruple is primitive exactly
when
$$
\gcd(m,Q)=1\qquad\text{and}\qquad m+n+p\equiv1\pmod 2.
$$
We also let $\delta=\gcd(n,p).$

Our aim is to determine the Frobenius number of $S$ explicitly.
The main result shows that, under the primitivity conditions above,
the formula
$$
g(S)=m^3+2(\delta-1)m^2-Qm+2mp\left(\frac{n}{\delta}-1\right)-2mn
$$
holds whenever $m\ge 2Q$ if $\delta=1,$ and, more generally, $m\ge \frac{2Q}{\delta}-1$ if $\delta>1.$

Two complementary descriptions are used to obtain this result.
We first study the Ap\'ery set with respect to $2mn$.  For the
stronger assumption $m\ge 2Q$, we construct an explicit complete
residue system modulo $2mn$ and associate to each residue class a
family of boundary representatives organized into a lift orbit.
The lift values become monotone away from the central levels, which
reduces the determination of the Ap\'ery representatives to finitely
many candidates.  This leads to an explicit maximal Ap\'ery element
and hence to the above Frobenius formula under $m\ge 2Q$.

Then we use a different argument to increase the range in which the
same formula is valid.  Writing $n=\delta n_0, p=\delta p_0,$
we introduce the auxiliary numerical semigroup $K=\langle m,n,p\rangle.$
Its residue-class structure with respect to $\delta$ gives an explicit
formula for $g(K)$.  A second residue decomposition, now modulo $2m$,
allows the gaps of $S$ to be controlled in terms of those of $K$.
This yields the weaker condition $m\ge \frac{2Q}{\delta}-1$ when $\delta>1$, and establishes the following main result of the paper.
\begin{theorem}
Let $m,n,p\in\mathbb Z_{>0}$, and let
$Q=n^{2}+p^{2}$ and $\delta=\gcd(n,p),$ for convenience. Assume that
$\gcd(m,Q)=1$ and $m+n+p\equiv 1\pmod 2.$ Now, let
$S=\left\langle 2mn,\, 2mp,\, m^{2}-Q,\, m^{2}+Q \right\rangle$ be the numerical semigroup generated by $2mn, 2mp, m^2-Q,$ and $m^2+Q.$ Suppose further that one of the following conditions holds:
\vskip 0.1in
(i) $\delta=1$ and $m \geq 2Q$.
\vskip 0.1in
(ii) $\delta>1$ and $m \geq \frac{2Q}{\delta}-1.$
\vskip 0.1in
Then we have
$$g(S)=m(m^{2}-Q)+ 2m\left((\delta-1)m+\frac{np}{\delta}-n-p\right). $$
\end{theorem}

The paper is organized as follows:  In Section 2, we establish the
primitivity criterion and develop the boundary-state and lift-orbit
framework.  In Section 3, we describe the Ap\'ery set and Frobenius number under the
assumption that $m\ge 2Q$ if $\delta = 1$, and $m\ge \frac{2Q}{\delta} - 1$ if $\delta > 1$.

\section{Primitive Pythagorean quadruples}
Throughout the paper, let $m,n,p$ be positive integers and
let $Q = n^2 + p^2.$

\begin{proposition}
\label{prop:primitive}
The quadruple $\left(2mn, 2mp, m^2 - Q, m^2 + Q\right)$ is primitive if and only if
$$
\gcd(m, Q) = 1 \qquad\text{and}\qquad m + n + p\equiv 1\pmod 2.
$$
\end{proposition}
\begin{proof}
%Note that $2mn$ and $2mp$ are even. Moreover, 
Since $x^2\equiv x\pmod 2$ for any integer $x$, we see that
$$
m^2 - Q \equiv m + n + p\pmod 2 ~~~\textrm{and} ~~~ m^2 + Q\equiv m + n + p\pmod 2.
$$
In particular, if $m + n + p$ is even, then the given quadruple cannot be primitive, and hence, to guarantee the primitiveness of the quadruple, it is necessary that
$$
m + n + p\equiv1\pmod2.
$$
Now, let $\ell$ be an odd prime dividing all the four entries of the quadruple. Then since $\ell$ divides both $2mn$ and $2mp,$
%$$
%\ell\mid 2mn,\qquad \ell\mid 2mp,
%$$
we have $\ell\mid mn$ and $\ell\mid mp$. If $\ell\nmid m$, then
$\ell\mid n$ and $\ell\mid p$, and hence, we get
$$
m^2 - Q\equiv m^2\not\equiv 0\pmod\ell,
$$
which is a contradiction. Thus, $\ell\mid m$. Then since $(m^2 + Q) - (m^2 - Q) = 2Q,$ we also have $\ell\mid Q$, whence $\ell \mid \gcd(m,Q).$
%Thus every odd common prime divisor divides $\gcd(m,Q)$.

Conversely, every odd prime dividing $\gcd(m,Q)$ divides all four entries of the quadruple, and hence, the quadruple is primitive exactly under the two stated conditions. 
\vskip 0.1in
This completes the proof.
\end{proof}

Now, for fixed positive integers $n$ and $p$, we call a positive integer $m$ \emph{admissible} if
$$
m^2>n^2+p^2,~~~
\gcd(m,n^2+p^2)=1,~~~\textrm{and}~~~
m+n+p\equiv1\pmod2.
$$
Thus, unless otherwise stated, all assertions involving $m$ are understood to be restricted to admissible values.

In the sequel, we assume that
$$
m\geq 2Q, ~~~ \gcd(m,Q) = 1,~~~\textrm{and}~~~ m + n + p\equiv 1\pmod2.
$$
In particular, note that $m^2 - Q > 0$.
\vskip 0.1in
Let
$$
\delta = \gcd(n,p),\qquad n = \delta n_0,\qquad p = \delta p_0,
$$
with $\gcd(n_0, p_0) = 1$. Since $\delta^2\mid Q$, we may write $Q=\delta^2 Q_0$ for some integer $Q_0$, and since $\gcd(m, Q) = 1$, we have $
\gcd(m, \delta) = 1.$ 
%Write
%$$
%Q = \delta^2Q_0.
%$$
Define
$$
A = 2mn, \qquad D = 2mp,\qquad B = m^2 - Q,\qquad C = m^2 + Q.
$$
We study the numerical semigroup
$$
S=\langle A, D, B, C\rangle.
$$

\begin{lemma}
\label{lem:elementary}
Under the standing assumptions,
\begin{equation}\label{eq:aux1}
Q - p(n_0 - 1) > 0
\end{equation}
and
\begin{equation}\label{eq:aux2}
m^2 - m(\delta - 1) - Q - p(n_0 - 1) > 0.
\end{equation}
Furthermore,
$$
m^2 - m(\delta - 1) + Q - p(n_0 - 1) > 0.
$$
\end{lemma}
\begin{proof}
The first inequality follows from
$$
p(n_0 -1) < pn_0 = \frac{np}{\delta}  \leq np  \leq \frac{n^2+p^2}{2} = \frac{Q}{2} < Q.
$$
For the second inequality, note that since $\delta - 1 \leq Q - 1$ and $p(n_0 - 1) < Q$, we have
$$
m^2 - m(\delta - 1) - Q - p(n_0 - 1) > m^2 - m(Q - 1) - 2Q.
$$
It remains to show that the last inequality holds. To this aim, let
$$
f(x) = x^2 - x(Q - 1) - 2Q.
$$
For $x\geq 2Q$, we have
$f'(x) = 2x - (Q - 1) > 0$ so that the function $f$ is increasing on $[2Q,\infty)$. It follows that
$$
f(m) \geq f(2Q) = 2Q^2 > 0,
$$
which proves the last inequality.
\vskip 0.1in
This completes the proof.
\end{proof}

\begin{lemma}
\label{lem:BC}
We have
$$
mB > (m-1)C.
$$
\end{lemma}
\begin{proof}
By using direct computation, we obtain
$$
mB - (m - 1)C = m^2 - (2m - 1)Q.
$$
The function
$$
g(x) = x^2 - (2x - 1)Q
$$
is increasing for $x\geq 2Q$, and hence, it follows that 
$$g(m) \geq g(2Q) = Q > 0.$$
This completes the proof.
%Therefore the statement holds.
\end{proof}

Now, for $x,y\in\N$, let $k=x+y$ and $z=y-x.$ Then
$x = \frac{k-z}{2}$ and $y = \frac{k+z}{2}$ so that
$$
k\geq |z|~~~\textrm{and}~~~  k\equiv z\pmod2.
$$
Moreover, we have
$$
xB + yC = km^2 + Qz.
$$
Also, let
$$
s = \frac{k-|z|}{2} = \min\{x, y\}.
$$
Then
$$
xB + yC = s(B+C) + |z|m^2 + Qz.
$$
Now, for three integers $v,h,z$ with
$
0\leq v < n_0$ and $ 0\leq h < \delta, $
let
$$
F(v,h,z) = vD + h(B+C) + |z|m^2 + Qz.
$$
Note that the expression for $F(v,h,z)$ is equivalent to
$$
F(v,h,z) = \begin{cases}
vD + h(B+C) + (-z)B, & z\leq 0,\\
vD + h(B+C) + zC, &z > 0.
\end{cases}
$$
Finally, we define the set of base states
$$
\mathcal B = \left\{(v,h,z):0\leq v < n_0, 0\leq h < \delta, -m\leq z < m\right\}.
$$
Note that $|\mathcal{B}|=n_0 \cdot \delta \cdot 2m = 2mn =A.$
\begin{lemma}
\label{lem:complete}
The set
$$
\{F(v,h,z):(v,h,z)\in\mathcal B\}
$$
is a complete residue system modulo $A = 2mn$.
\end{lemma}
\begin{proof}
Note first that $|\{F(v,h,z):(v,h,z) \in \mathcal{B} \}| = |\mathcal{B}|=A.$ The indexing set has cardinality
$
n_0\cdot \delta\cdot 2m = 2mn = A,
$
and hence, it is sufficient to prove injectivity modulo $A$.
Suppose that
$$
F(v,h,z)\equiv F(v',h',z')\pmod{A}.
$$
Reducing modulo $m$, we obtain $$Q(z-z')\equiv 0\pmod{m}.$$ 
Since $\gcd(m, Q) = 1$, it follows that $z \equiv z^{\prime} \pmod{m}.$
Then since $z, z'\in[-m, m)$, we see that $z - z' \in \{-m, 0, m \}.$
\vskip 0.1in
Assume first that $z - z' = m$. Then
$0\leq z < m, -m\leq z' < 0$, and $|z| - |z'| = 2z - m.$
Thus
\begin{equation}\label{eqn 1}
F(v, h, z) - F(v', h', z') = m\left(2\delta p_0(v - v') + 2m(h - h') + m(2z - m) + Q\right).    
\end{equation}
Since $F(v,h,z)-F(v',h',z')$ is divisible by $A = 2m\delta n_0$, the expression in
parentheses on the right hand side of (\ref{eqn 1}) should be even. However, since
$$
m^2 + Q\equiv m + n + p\equiv 1\pmod{2},
$$
it is absurd. Analogously, we can exclude the case when $z - z' = -m$. Hence $z = z'$.

The congruence now reduces to
$$
2m\delta p_0(v - v') + 2m^2(h - h')\equiv 0\pmod{2m\delta n_0}.
$$
By dividing both sides by $2m$, we obtain
$$
\delta p_0(v - v') + m(h - h') \equiv 0\pmod{\delta n_0}.
$$
It implies that $m(h - h')\equiv 0\pmod{\delta},$ and then since $\gcd(m, \delta) = 1$, it follows that
$
h\equiv h'\pmod{\delta}.
$
Therefore the range of $h - h'$ gives $h = h'$, and then we have
$
p_0(v - v')\equiv 0\pmod{n_0}.
$
Since $\gcd(p_0, n_0) = 1$, we obtain $v = v'$.
\vskip 0.1in
This comletes the proof.
\end{proof}
\begin{lemma}
\label{lem:normalization}
Let $N\in S$. Then there exists a boundary state
$$
(v_0, h, z),\qquad 0\leq v_0 < n_0,\quad 0\leq h < \delta,\quad z\in\mathbb{Z},
$$
such that
$$
F(v_0, h, z)\equiv N\pmod{A}
$$
and
$$
F(v_0, h, z)\leq N.
$$
\end{lemma}
\begin{proof}
Write $N$ as a nonnegative integer combination of $A,B,C,$ and $D$ so that we have
$$
N = uA + vD + xB + yC
$$
with $u, v, x, y\in\N$. Since $uA$ does not affect the residue modulo
$A$, it may be omitted.
\vskip 0.1in
Firstly, we reduce $v$ modulo $n_0$. Write
$$
v = v_1 + n_0q,\qquad 0\leq v_1 < n_0,\quad q\in\N.
$$
Since
$
n_0D = p_0A,
$
we have
$$
vD = v_1D + qp_0A.
$$
Thus replacing $v$ by $v_1$ decreases the value by a nonnegative
multiple of $A$.

Now let
$$
s = \min\{x, y\},
\qquad
z = y - x.
$$
Then
$$
xB + yC = s(B + C) + |z|m^2 + Qz.
$$
Write
$$
s = h + \delta r,\qquad 0\leq h < \delta,\quad r\in\N.
$$
Choose the unique $v_0\in\{0,\dots,n_0 - 1\}$ satisfying
$$
p_0(v_1 - v_0) + mr\equiv 0\pmod{n_0}.
$$
Then
$$
v_1D + s(B + C) - \left(v_0D + h(B + C)\right) = 2m\delta\left(p_0(v_1 - v_0) + mr\right),
$$
which is a multiple of $A = 2m\delta n_0$.

Now, let us check that $p_0(v_1 - v_0) + mr$ is nonnegative. If $r = 0$, then
$$
p_0(v_1 - v_0)\equiv 0\pmod{n_0},
$$
and it implies that $v_1 = v_0$. If $r\geq 1$, then
$$
p_0(v_1 - v_0) + mr \geq - p_0(n_0 - 1) + m>0,
$$
because $m\geq 2Q > p_0(n_0 - 1)$. Hence 
$$N-F(v_0,h,z)=uA+qp_0 A + 2m\delta(p_0(v_1 -v_0)+mr) $$
is a nonnegative multiple of $A$, and hence, we conclude that $F(v_0,h,z) \leq N,$ as desired. 
%\hwang{맞나유?}\song{예압}
\vskip 0.1in
This completes the proof.
\end{proof}
%\hwang{26.08.29}

Let $\sigma_0 = (v_0, h_0, z_0)\in\mathcal B$ and $\ell\in\mathbb{Z}$. Let $z_{\ell} = z_0 + 2m\ell$
and
$a_{\ell} = \frac{|z_{\ell}|-|z_0|}{2}.$ Let $h_{\ell}$ be the unique integer satisfying
$$
0\leq h_{\ell} < \delta ~~~\textrm{and}~~~ h_{\ell} \equiv h_0 - a_{\ell} \pmod{\delta}.
$$
Furthermore, let us define
$$
r_{\ell} = \frac{a_{\ell} + h_{\ell} - h_0}{\delta}.
$$
Finally, let $v_{\ell}$ be the unique integer satisfying
$$
0\leq v_{\ell} < n_0
$$
and
$$
p_0(v_{\ell} - v_0) + mr_{\ell} + \delta Q_0\ell
\equiv 0\pmod{n_0}.
$$
In this situation, we may introduce the following notion. 
\begin{definition}
The state
$\sigma_{\ell} = (v_{\ell}, h_{\ell}, z_{\ell})$ is called the \emph{$\ell$-th lift} of $\sigma_0$.
\end{definition}

%\hwang{normal 한 def 스타일이 아니긴해, 내가 지금 써놓은대로 하면 어떠심?}\song{동의함다}
\begin{lemma}
\label{lem:residue-invariance}
For every $\ell\in\mathbb{Z}$, we have
$$
F(\sigma_{\ell})\equiv F(\sigma_0)\pmod A.
$$
More precisely,
$$
F(\sigma_{\ell}) - F(\sigma_0) = 2m\delta\left(p_0(v_{\ell} - v_0) + mr_{\ell} + \delta Q_0\ell \right).
$$
\end{lemma}
\begin{proof}
Using
$$
|z_{\ell}|-|z_0|=2a_{\ell},\qquad z_{\ell} - z_0 = 2m\ell,
$$
and
$$
a_{\ell} + h_{\ell} - h_0 = \delta r_{\ell},
$$
we obtain
\begin{align*}
	F(\sigma_{\ell}) - F(\sigma_0)
	& = (v_{\ell} - v_0)D + (h_{\ell} - h_0)(B + C) +\left(|z_{\ell}| - |z_0|\right)m^2 + Q(z_{\ell} - z_0)\\
	& =	2m\delta p_0(v_{\ell} - v_0) + 2\delta m^2r_{\ell} + 2m\delta^2Q_0\ell \\
    & = 2m \delta \left(p_0(v_{\ell}-v_{0}) + m r_{\ell} + \delta Q_0 \ell \right),
\end{align*}
which is the stated formula. By the definition of $v_{\ell}$, the quantity in parentheses is divisible by $n_0$, and hence, the difference $F(\sigma_{\ell})-F(\sigma_0)$ is divisible by $A$.
\vskip 0.1in
This completes the proof.
\end{proof}
%\hwang{2026.09.02 완료}
\begin{lemma}
\label{lem:reconstruction}
Every boundary state
$$
(v, h, z),\qquad 0\leq v < n_0,\quad 0\leq h < \delta,\quad z\in\mathbb{Z},
$$
is the $\ell$-th lift of a unique base state
$$
(v_0, h_0, z_0)\in\mathcal B
$$
for a unique $\ell\in\mathbb{Z}$.
\end{lemma}
\begin{proof}
Choose the unique $z_0 \in [-m,m)$ and $\ell \in \mathbb{Z}$ such that $z = z_0 + 2m\ell,$ and let
$a_{\ell} = \frac{|z|-|z_0|}{2}.$ Then choose the unique $h_0 \in \{0,\dots,\delta-1\}$ satisfying
$$
h_0\equiv h + a_{\ell}\pmod\delta,
$$
and let
$
r_{\ell} = \frac{a_{\ell} + h - h_0}{\delta}.
$
Finally, choose the unique $v_0 \in \{0,\dots,n_0 - 1\}$ satisfying
\begin{equation}\label{eqn 2}
p_0(v - v_0) + mr_{\ell} + \delta Q_0\ell \equiv 0\pmod{n_0}.
\end{equation}

By construction, $(v, h, z)$ is the $\ell$-th lift of $(v_0,h_0,z_0)$, and the uniqueness of $z_0$ and $\ell$ determines $a_{\ell}$, and hence, $h_0$ and $r_{\ell}$ uniquely. Since $p_0$ is invertible modulo $n_0,$ the congruence (\ref{eqn 2}) determines $v_0$ uniquely modulo $n_0.$ 

\vskip 0.1in
This completes the proof.
\end{proof}

%Since the representation of $z_0$ modulo $2m$ is unique, it guarantees the uniqueness of $h_0$ modulo $\delta$ and therefore it implies that $v_0$ modulo $n_0$ is also unique.

\begin{corollary}
\label{cor:orbit}
Two boundary elements are congruent modulo $A$ if and only if they
belong to the same lift orbit.
\end{corollary}
\begin{proof}
Elements in one lift orbit are congruent by
Lemma \ref{lem:residue-invariance}. Conversely, reconstruct both
boundary elements from their base states. Their base elements are
congruent modulo $A$, so the complete residue system
Lemma \ref{lem:complete} implies that the base states coincide.
\end{proof}

For a fixed base state
$\sigma_0=(v_0,h_0,z_0)\in\mathcal B,$
the integer $\ell$ indexing the lift $\sigma_{\ell}$ will be called the \emph{lift level}. We refer to $\sigma_\ell$ as a
\emph{positive lift} if $\ell>0$, as a \emph{negative lift} if $\ell<0$,
and as the \emph{base state} if $\ell=0$.  Thus, moving in the positive
direction along a lift orbit means increasing the lift level $\ell$,
whereas moving in the negative direction means decreasing $\ell$.
Equivalently, a positive step replaces
$z_\ell$ by $z_{\ell+1}=z_\ell+2m,$
while a negative step replaces it by $z_{\ell-1}=z_\ell-2m.$

These terms refer only to the direction of motion along the lift orbit;
they do not refer to the sign of the corresponding value $F_\ell:=F(\sigma_\ell),$ nor necessarily to the sign of $z_\ell$.

\begin{lemma}
\label{lem:delta-r}
For consecutive positive lifts and for consecutive negative lifts,
the corresponding difference $\Delta r$ satisfies
$$
\Delta r \in \left\{\left\lfloor\frac m\delta\right\rfloor, \left\lceil\frac m\delta\right\rceil\right\}.
$$
In particular,
$$
\Delta r\geq 2\delta Q_0.
$$
\end{lemma}

\begin{proof}
In either direction, $a$ increases by $m$ between
consecutive lifts. Hence we have
$$
\Delta r = \left\lceil\frac{u+m}{\delta}\right\rceil - \left\lceil\frac u\delta\right\rceil 
$$
for some integer $u$.
Write $m = q\delta+\rho$ for some $0\leq\rho < \delta.$ If $\rho =0,$ then $\Delta r = q,$ and if $0< \rho <\delta,$ then by direct computation, we get $\Delta r\in\{q,q+1\}$. Hence in all cases, we obtain that $\Delta r \in \left\{\left\lfloor\frac m\delta\right\rfloor, \left\lceil\frac m\delta\right\rceil\right\}.$
\vskip 0.1in
For the second assertion, note that since
$m\geq 2Q = 2\delta^2Q_0,$ we have
$\frac m\delta\geq 2\delta Q_0.$ Then since $2 \delta Q_0$ is an integer, it follows that
$\left\lfloor\frac m\delta\right\rfloor\geq 2\delta Q_0,$ and the desired result follows from the first assertion. 
%and the inequality $\left\lfloor\frac m\delta\right\rfloor \leq \left\lceil\frac m\delta\right\rceil$.
\vskip 0.1in
This completes the proof.
\end{proof}

%$$
%F_{\ell} = F(\sigma_{\ell}).
%$$

The next result describes how the value $F_\ell=F(\sigma_\ell)$ changes as one moves along a lift orbit. As above, moving in the positive direction means increasing the lift level $\ell$, whereas moving in the negative direction means decreasing $\ell$.

\begin{lemma}
\label{lem:quotients}
For every $\ell\geq 1$, there exists an integer $t_{\ell}^+ \geq 1$ such that
$$
F_{\ell + 1}-F_{\ell} = A t_{\ell}^{+}.
$$
For every $j\geq 1$, there exists an integer $t_j^- \geq 1$ such that
$$
F_{-(j + 1)}-F_{-j} = A t_j^{-}.
$$
\end{lemma}
\begin{proof}
For positive lifts, subtracting the defining congruences for $v_{\ell + 1}$
and $v_{\ell}$ gives
$$
p_0(v_{\ell + 1} - v_{\ell}) + m(r_{\ell + 1} - r_{\ell}) + \delta Q_0 = n_0t_{\ell}^{+}
$$
for some $t_{\ell}^{+}\in\mathbb{Z}$. By a direct computation, we have
$$
F_{\ell + 1} - F_{\ell} = A t_{\ell}^{+}.
$$
Now
$$
v_{\ell + 1} - v_{\ell} \geq -(n_0 - 1)
$$
and, by Lemma \ref{lem:delta-r},
$$
r_{\ell + 1} - r_{\ell} \geq 2\delta Q_0.
$$
Furthermore,
$$
p_0(n_0 - 1)
<
p_0n_0
=
\frac{np}{\delta^2} \leq \frac{n^2+p^2}{2\delta^2} = \frac{Q_0}{2}.
$$
Therefore
\begin{equation}\label{eq:positive-lift}
n_0 t_{\ell}^{+} \geq -p_0(n_0 - 1) + 2m\delta Q_0 + \delta Q_0 > -\frac{Q_0}{2} + \delta Q_0(2m + 1) > 0,
\end{equation}
and hence, we get $t_{\ell}^{+}\geq 1$.
\vskip 0.1in
Now, we consider the case for the negative direction. There exists an integer $t_j^{-}$ such that
$$
p_0(v_{-(j + 1)} - v_{-j}) + m(r_{-(j + 1)}-r_{-j}) -\delta Q_0 = n_0 t_j^{-},
$$
and
$$
F_{-(j + 1)} - F_{-j} = A t_j^{-}.
$$
Note that we can use the same bounds in (\ref{eq:positive-lift}) as follows:
$$
n_0t_j^{-} \geq -p_0(n_0 - 1) + 2m\delta Q_0 - \delta Q_0 > -\frac{Q_0}{2} + \delta Q_0(2m - 1) > 0.
$$
Hence $t_j^{-} \geq 1$.
\end{proof}

\begin{corollary}
\label{cor:finite}
For every base state $\sigma_0$, we have
$$
F_1 < F_2 < F_3 < \cdots
$$
and
$$
F_{-1} < F_{-2} < F_{-3} < \cdots.
$$
Consequently, we have
$$
\min_{\ell\in\mathbb{Z}}F_{\ell} = \min\{F_{-1},F_0,F_1\}.
$$
\end{corollary}
\begin{remark}
    Note that if, in addition, $F_{-1}>F_0 <F_1,$ then Corollary \ref{cor:finite} shows that $F_0$ is the unique global minimum and the sequence $(F_{\ell})_{\ell \in \mathbb{Z}}$ is strictly decreasing toward $F_0$ from the left and strictly increasing from $F_0$ to the right. Equivalently, the sequence $(-F_{\ell})_{\ell \in \mathbb{Z}}$ is strictly unimodal with a unique maximum at $\ell =0.$ Hence, by Corollary \ref{cor:finite}, the unimodality of $(-F_{\ell})_{\ell \in \mathbb{Z}}$ reduces to the two central inequalities $F_{-1}>F_0$ and $F_1>F_0.$ It would be interesting to determine whether this unimodality of $(-F_{\ell})_{\ell \in \mathbb{Z}}$ always holds.
\end{remark}

%\hwang{Cor 2.9 표현이 수학적이지 않은거 빼고, 고침 26.09.03}
\section{The Apéry set}

\begin{proposition}
\label{prop:apery}
Let $r\in\mathbb{Z}/A\mathbb{Z}$, and let $\sigma_r\in\mathcal B$ be
the unique base state satisfying
$$
F(\sigma_r)\equiv r\pmod{A}.
$$
Then the smallest element of $S$ congruent to $r$ modulo $A$ is
$$
w_r = \min_{\ell\in\mathbb Z}F_{\ell}(\sigma_r) = \min\{F_{-1}(\sigma_r),F_0(\sigma_r),F_1(\sigma_r)\}.
$$
In particular, we have
$$
\Ap(S, A) = \left\{\min_{\ell\in\mathbb{Z}}F_\ell(\sigma): \sigma\in\mathcal B\right\}.
$$
\end{proposition}
\begin{proof}
Every lift value belongs to $S$ and is congruent to the base value modulo $A$, and hence, the orbit minimum is an admissible representative. %\hwang{admissible representative? definition 이 어디?}\song{Page 2에 추가했심다}

Conversely, let $N\in S$ have residue $r$ modulo $A$. By Lemma \ref{lem:normalization}, there exists a boundary state $\sigma$ such that $N=F(\sigma)+qA$ for some $q\in\mathbb Z_{\ge0}$. In particular, $F(\sigma)$ lies in the same residue class modulo $A$ as $N$ and satisfies $F(\sigma)\le N$.

a boundary element $N_0$ in the same residue class. By Corollary \ref{cor:orbit}, $N_0$ belongs to the lift orbit of $\sigma_r$. Hence it follows that
$$
N\geq N_0\geq \min_{\ell\in\mathbb{Z}}F_{\ell}(\sigma_r),
$$
and thus, the orbit minimum is exactly the Ap$\acute{e}$ry representative. The last equality follows from Corollary \ref{cor:finite}.
\vskip 0.1in
This completes the proof.

\end{proof}

Recall that $F(v,h,z) = vD+h(B+C)+|z|m^2+Qz,$ where
$$
0\le v<n_0,\qquad 0\le h<\delta,\qquad -m\le z<m.
$$
Since $D>0$ and $B+C>0$, the contribution $vD+h(B+C)$ is largest when $(v,h)=(n_0-1,\delta-1).$
On the negative part of the base interval, $-m\le z\le0$, the remaining
contribution is $|z|m^2+Qz=(-z)B,$ which is largest at $z=-m$, where its value is $mB$. This suggests the base state $\sigma_=(n_0-1,\delta-1,-m)$ and the corresponding candidate $W:=F(\sigma)=(n_0-1)D+(\delta-1)(B+C)+mB.$
We will prove below that $W$ is indeed the maximum element of the Ap\'ery
set.

Now, let 
$$
W = (\delta - 1)(B + C) + (n_0 - 1)D + mB.
$$
%\hwang{종도형 경우와 마찬가지로, 이렇게 정의하는 자연스러운 이유가 있는건가?}\song{위에 추가 했심다.}
\begin{lemma}
\label{lem:base-max}
For every base state $\sigma\in\mathcal B$, we have $F(\sigma)\leq W.$

\end{lemma}
\begin{proof}
Since the coefficients of $v$ and $h$ are positive, the maximum occurs at $v=n_0 -1$ and $h = \delta -1.$
%$$
%v = n_0 - 1,\qquad h = \delta - 1.
%$$
For $-m\leq z\leq 0$,
$$
|z|m^2 + Qz = (-z)B,
$$
whose maximum is $mB$ at $z = -m$.
For $1\leq z < m$,
$$
|z|m^2 + Qz = zC,
$$
whose maximum is $(m-1)C$ at $z = m - 1$.
\vskip 0.1in
Now, since $mB >(m-1)C$ by Lemmar \ref{lem:BC}, we can conclude that the maximum base value equals $W$.
\vskip 0.1in
This completes the proof.
\end{proof}

\begin{lemma}
\label{lem:W-apery}
The element $W$ is the smallest element of $S$ in its residue class modulo $A$. Consequently, we have
$W\in\Ap(S, A).$
\end{lemma}
\begin{proof}
The element \(W\) corresponds to the base state $\sigma_*=(n_0-1,\delta-1,-m).$
Let $F^*_\ell$ denote the value of its $\ell$-th lift.  Then $F^*_0=W$.

For the first positive lift, $z$ changes from $-m$ to $m$, so that $a_1=0$, and hence, it follows that $h_1=h_0=\delta-1$. Hence we get
$$
F^*_1-W=(v_1-(n_0-1))D+2mQ.
$$
Since $v_1-(n_0-1)\ge-(n_0-1)$,
$$
F^*_1-W
\ge2m\bigl(Q-p(n_0-1)\bigr)>0
$$
by \eqref{eq:aux1}.

For the first negative lift, $z$ changes from $-m$ to $-3m$.  Hence
$$
F^*_{-1}-W
=(v_{-1}-(n_0-1))D
 +(h_{-1}-(\delta-1))(B+C)+2m^3-2mQ.
$$
Using
$$
v_{-1}-(n_0-1)\ge-(n_0-1)
~~~\textrm{and}~~~
h_{-1}-(\delta-1)\ge-(\delta-1),
$$
we get
$$
F^*_{-1}-W
\ge2m\bigl(m^2-m(\delta-1)-Q-p(n_0-1)\bigr)>0
$$
by \eqref{eq:aux2}. Now, Corollary \ref{cor:finite} shows that every
other nonzero lift is even larger.  Thus $F^*_0=W$ is the orbit minimum,
and Proposition~\ref{prop:apery} gives $W\in\Ap(S,A)$.
\vskip 0.1in
This completes the proof.
\end{proof}

\begin{theorem}
\label{thm:max-apery}
Under the standing assumption $m\ge2Q$, we have $\max\Ap(S, A) = W.$
\end{theorem}
\begin{proof}
For any residue class, its Ap$\acute{e}$ry representative is the minimum of the
corresponding lift orbit, and therefore is at most the base value.
By Lemma \ref{lem:base-max}, every base value is at most $W$. Hence
$$
\max\Ap(S, A)\leq W.
$$
On the other hand, Lemma \ref{lem:W-apery} gives $W\in\Ap(S, A),$ and hence, the desired equality holds.
\vskip 0.1in
This completes the proof.
\end{proof}
%\hwang{26.09.03 완료}

\begin{theorem}\label{thm:uniform-formula}
Let $m,n,p$ be positive integers, let $\delta=\gcd(n,p)$, and assume that
$$
m\ge2(n^2+p^2),~~~
\gcd(m,n^2+p^2)=1,~~\textrm{and}~~~
m+n+p\equiv1\pmod2.
$$
Then
\begin{equation}\label{eq:formula}
 g(S)=m^3+2(\delta-1)m^2-(n^2+p^2)m
 +2mp\left(\frac n\delta-1\right)-2mn.
\end{equation}
\end{theorem}

\begin{proof}
By Theorem \ref{thm:max-apery} and \eqref{eq:apery-frob},
$$
g(S)=W-A.
$$
Since
$$
B+C=2m^2,
\qquad D=2mp,
\qquad B=m^2-Q,
$$
we obtain
\begin{align*}
W & = (\delta-1)(B+C)+(n_0-1)D+mB\\
& = 2(\delta-1)m^2+2mp(n_0-1)+m(m^2-Q)\\
& = m^3+2(\delta-1)m^2-Qm+2mp(n_0-1).
\end{align*}
Subtracting $A=2mn$, and using $n_0=n/\delta$, we obtain
\eqref{eq:formula}.
\vskip 0.1in
This completes the proof.
\end{proof}

\begin{corollary}
If $\gcd(n,p) = 1$, then
$$
g(S) = m^3 - (n^2 + p^2)m + 2mp(n - 1) - 2mn.
$$
\end{corollary}
\begin{proof}
    In this case, $\delta = 1,$ and hence, the desired result follows from Theorem \ref{thm:uniform-formula}.
\end{proof}
\begin{corollary}
Let $n = p = t$. If
$$
m\geq 4t^2,\qquad m\text{ is odd},\qquad \gcd(m, t) = 1,
$$
then
$$
g\left(2mt, m^2 - 2t^2, m^2 + 2t^2\right) = m^3 + 2(t - 1)m^2 - 2t(t + 1)m.
$$
\end{corollary}
\begin{proof}
    In this case, $\delta = t=n,$ and $Q=2t^2,$ and hence, the desired result follows from Theorem \ref{thm:uniform-formula}.
\end{proof}

Now, we show that when $\delta > 1$, the same formula in Theorem \ref{thm:uniform-formula} is still valid under the slightly weaker condition that
$m\geq \frac{2Q}{\delta} - 1.$
\begin{theorem}
Let $m,n,p\in\mathbb Z_{>0}$, and let
$Q=n^{2}+p^{2}$ and $\delta=\gcd(n,p),$ for convenience. Assume that
$\gcd(m,Q)=1$ and $m+n+p\equiv 1\pmod 2.$ Now, let
$S=\left\langle 2mn,\, 2mp,\, m^{2}-Q,\, m^{2}+Q \right\rangle$ be the numerical semigroup generated by $2mn, 2mp, m^2-Q,$ and $m^2+Q.$ Suppose further that one of the following conditions holds:
\vskip 0.1in
(i) $\delta=1$ and $m \geq 2Q$.
\vskip 0.1in
(ii) $\delta>1$ and $m \geq \frac{2Q}{\delta}-1.$
\vskip 0.1in
Then we have
$$g(S)=m(m^{2}-Q)+ 2m\left((\delta-1)m+\frac{np}{\delta}-n-p\right). $$
\end{theorem}

\begin{proof}
Since $\delta=\gcd(n,p),$ we may write $n=\delta u$ and $p=\delta v$ for some relatively prime integers $u,v.$ Let $H=\langle u,v \rangle$ and $K=\langle m,n,p \rangle =\langle m, \delta u, \delta v \rangle.$
Since $Q=n^2 + p^2 = \delta^2 (u^2 +v^2)$ and $\gcd(m,Q)=1,$ we also have $\gcd(m,\delta)=1.$
\vskip 0.1in
Now, We first show that $m \in H.$ 
\vskip 0.1in
(Case (i)) If $\delta=1,$ then by a well-known fact, $g(H)=uv-u-v = np-n-p < Q.$ Since $m \geq 2q,$ it follows that $m>g(H),$ and hence, $m \in H.$
\vskip 0.1in
(Case (ii)) If $\delta>1,$ then since $g(H)=uv-u-v < u^2 + v^2$ and $m \geq \frac{2q}{\delta}-1 = 2d(u^2 +v^2 )-1  \geq u^2+v^2$, we also have $m > g(H)$ and $m \in H.$
\vskip 0.1in
Now, since $m \in H$ and $\gcd(m,\delta)=1,$ we can write 
$$K=\bigsqcup_{j=0}^{\delta-1}
\left(jm+\delta \cdot H\right)$$
where the union is a disjoint union. It follows that 
$$ g(K)=(\delta-1)m + \delta \cdot g(H)  = (\delta-1)m+\delta(uv-u-v) =(\delta-1)m+\frac{np}{\delta}-n-p. $$
Since either $m \geq 2Q$ (in Case (i)) or $m \geq \frac{2Q}{\delta}-1$ (in Case (ii)), we have $m^2-Q> g(K)$ so that $m^2 - Q \in K$, and hence, we get $K=\langle m,n,p,m^2-Q \rangle.$ Also, since $m+n+p$ is odd, $m^2-Q$ is odd, and we have
$$\gcd(m,m^2 -Q)=\gcd(m,Q)=1,$$
and it follows that $\gcd(m^2 -Q, 2m)=1.$ 
\vskip 0.1in
Now, for $0 \leq r < 2m,$ let 
$\Gamma_r = \max\{t \in \mathbb{Z}~|~t \not \in K~\textrm{and}~t-(m^2+Q-rm) \not \in K \}.$ Then we get that
$$S \cap (r(m^2-Q)) = \{r (m^2 -Q) + 2m t~|~ t \in K \cup (m^2+Q-rm+K)\}.$$
Hence it follows that $g(S)=\max_{0 \leq r < 2m} \{r(m^2-Q)+2m \cdot \Gamma_r \}.$ Now, there are two cases to consider:
\vskip 0.1in
(Case 1): If $0 \leq r \leq m,$ then since $Q=n^2 + p^2 \in K,$ we get that 
$$m^2 +Q - rm = m(m-r)+Q \in K$$
and in this case, we obtain that $\Gamma_r = g(K).$ (Thus, the largest contribution from (Case 1) is $m(m^2 -Q) + 2m \cdot g(K).$)
\vskip 0.1in
(Case 2): If $m+1 \leq r <2m,$ then write $r=m+s$ for some $1 \leq s <m,$ and then, we have
$$m^2 + Q - rm = m^2 +Q - (m+s)m = Q-sm.$$
Now, if $\delta=1$ and $m \geq 2Q,$ then we have $\Gamma_r = \Gamma_{m+s}\leq g(K)-(sm-Q)$ and 
$$2m(sm-Q)-s(m^2-Q) = s(m^2+Q) -2mQ \geq m^2 +Q - 2mQ  \geq 0.$$
Thus we see that 
$$ r(m^2 -Q) + 2m \cdot \Gamma_{r} = (m+s)(m^2 -Q) + 2m \cdot \Gamma_{m+s} \leq m (m^2 -Q) + 2m \cdot g(K).$$
Now, suppose that $\delta>1$ and $m \geq \frac{2q}{\delta}-1.$ For $1 \leq s <\delta,$ we have
$\Gamma_r = \Gamma_{m+s} \leq g(K)-\min \{sm, \delta m-Q \},$ and for $\delta \leq s \leq m-1,$ we have $\Gamma_r = \Gamma_{m+s} \leq g(K)-(sm-Q).$ In the first case, we clearly have $2m\cdot sm \geq s(m^2-Q)$ and then since $m \geq \frac{2Q}{\delta}-1,$ we also have
$$2m(\delta m-Q)-(\delta-1)(m^2-Q) = (\delta+1)m^2 -2mQ + (\delta-1)Q \geq 0.$$
Then since $1 \leq s \leq \delta-1,$ it follows that 
$$2m \cdot \min \{sm, \delta m-Q \} \geq s (m^2 -Q).$$
In the second case, since $\delta \leq s $ and $m \geq \frac{2Q}{\delta}-1$, we have 
$$2m(sm-Q)-s(m^2-Q) = s(m^2 +Q)-2mQ \geq \delta (m^2 +Q)-2mQ  \geq 0.$$
Therefore, for any $1 \leq s \leq m-1,$ we have
$$(m+s)(m^2-Q)+2m \cdot \Gamma_{m+s} \leq m(m^2-Q)+2m \cdot g(K),$$
and hence, we can conclude that $g(S)=m(m^2-Q)+2m \cdot g(K).$ Then since $g(K)=(\delta-1)m+\frac{np}{\delta}-n-p,$ we finally have
$$g(S)=m(m^2-Q) +2m \left((\delta-1)m + \frac{np}{\delta}-n-p \right).$$

This completes the proof.
\end{proof}

\end{document}